\documentclass{amsart}

\usepackage{amssymb,color}
\RequirePackage[colorlinks,citecolor=blue,urlcolor=blue,pagebackref=false]{hyperref}

\numberwithin{equation}{section}

\newtheorem{theorem}{Theorem}
\newtheorem{lemma}[theorem]{Lemma}
\newtheorem{corollary}[theorem]{Corollary}

\theoremstyle{definition}

\newtheorem{example}[theorem]{Example}

\theoremstyle{remark}

\newcommand{\E}{\mathbb{E}}

\newcommand{\R}{\mathbb{R}}
\def\var{\operatorname{Var}}		
\def\dx{\operatorname{d}}

\begin{document}

	\title[Weak laws of large numbers
	for maximal partial sums]{On weak laws of large numbers
		for maximal partial sums of pairwise independent random variables}
	
	
	\author[L. V. Th\`{a}nh]{L\^{e} V\v{a}n Th\`{a}nh}
	\address{Department of Mathematics, Vinh University, 182 Le Duan, Vinh, Nghe An, Vietnam}
	\email{levt@vinhuni.edu.vn}

	
	\date{}
	
	\begin{abstract}
		This paper develops Rio's method [C. R. Acad. Sci. Paris S\'{e}r. I Math., 1995] to
		prove the weak law of
		large numbers 
		for maximal partial sums of pairwise independent random variables. 
		The method allows us to avoid using the Kolmogorov
		maximal inequality. As an application,
	a weak law of large numbers for maximal partial sums
		of pairwise independent random variables
		under a uniform integrability condition is also established.
The sharpness of the result is illustrated by an example.

	\end{abstract}
	
	\maketitle
	
	\section{Introduction and results}
	
	A real-valued function $L(\cdot )$ is said to be \textit{slowly varying} (at infinity) if it is 
	a positive and measurable function on $[A,\infty)$ for some $A\ge0$, and for each $\lambda>0$,
	\begin{equation*}\label{rv01}
		\lim_{x\to\infty}\dfrac{L(\lambda x)}{L(x)}=1.
	\end{equation*}
	In \cite{bruijn1959pairs}, de Bruijn proved that if $L(\cdot)$ is a slowly varying function, then
	there exists a slowly varying function $\tilde{L}(\cdot)$, unique up to asymptotic equivalence, satisfying
	\begin{equation*}\label{BGT1513}
		\lim_{x\to\infty}L(x)\tilde{L}\left(xL(x)\right)=1\ \text{ and } \lim_{x\to\infty}\tilde{L}(x)L\left(x\tilde{L}(x)\right)=1.
	\end{equation*}
	The function $\tilde{L}(\cdot)$ is called the de Bruijn conjugate of $L(\cdot)$ (\cite[p. 29]{bingham1989regular}).
	Bojani\'{c} and Seneta \cite{bojanic1971slowly} showed that for most of ``nice'' slowly varying functions,
	we can choose $\tilde{L}(x)=1/L(x)$.
	Especially, if  $L(x)=\log^\gamma(x)$ or $L(x)=\log^\gamma(\log(x))$ for some $\gamma\in\R$, then $\tilde{L}(x)=1/L(x)$.
		Here and thereafter, for a real number $x$, $\log(x)$
	denotes the natural logarithm (base $\mathrm{e}$) of $\max\{x,\mathrm{e}\}$.

Let $L(\cdot)$ be a slowly varying function and let $r>0$. By using a suitable asymptotic equivalence version (see Lemma 2.2 and Lemma 2.3 (i) in Anh et al. \cite{anh2021marcinkiewicz}), 
we can assume that $L(\cdot)$ is positive and
differentiable on $[a,\infty)$, and $x^r L(x)$ is strictly increasing on $[a,\infty)$  for some large $a$. Next, let
	$L_1(\cdot)$ be a slowly varying function with
	$L_1(0)=0$ with a linear growth to $L(a)$
	over $[0,a)$, and $L_1(x)\equiv L(x)$ on $[a,\infty)$.
	Then (i) $L_1(x)$ is continuous on $[0,\infty)$ and differentiable on $[a,\infty)$, and
	(ii) $x^r L_1(x)$ is strictly increasing on $[0,\infty)$.
	In this paper, we will assume, without loss of generality, that properties (i) and (ii) are fulfilled for all slowly varying functions.
	
	The starting point of the current research is the following weak law of large numbers (WLLN) which was proved by Gut \cite{gut2004extension}. Hereafter, $\mathbf{1}(A)$ denotes the indicator function of
	a set $A$.
	
	\begin{theorem}[Gut \cite{gut2004extension}]\label{thm.gut}
		Let $0< p\le 1$ and let
		$\{X,X_n, \, n \geq 1\}$ be a sequence of independent identically distributed (i.i.d.) random variables. Let $L(\cdot)$ be a slowly varying function
		and let $b_x=x^{1/p}L(x),\ x\ge0$.
		Then
		\begin{equation}\label{eq.gut.main03}
			\begin{split}
				\dfrac{\sum_{i=1}^{n} X_i-n\E\left(X\mathbf{1}(|X|\le b_n)\right)}{b_n}\overset{\mathbb{P}}{\to}0\ \text{ as }n\to\infty
			\end{split}
		\end{equation} 
		if and only if
		\begin{equation}\label{eq.gut.main05}
			\lim_{n\to \infty}n\mathbb{P}(|X|>b_n)=0.
		\end{equation}
	\end{theorem}
	The above WLLN
	has been extended in
	several directions, see \cite{rosalsky2006weak,rosalsky2009weak} for 
	WLLNs with random indices for arrays of independent random variables
	taking values in
	Banach spaces, and see \cite{boukhari2021weak,chandra2014extension,hien2015weak,kruglov2011generalization} and the references therein
	for WLLNs for dependent random variables and dependent random vectors. Boukhari \cite[Theorem 1.2]{boukhari2021weak} showed that for $0<p<1$, condition \eqref{eq.gut.main05}
	implies
	\begin{equation}\label{boukhari03}
		\begin{split}
			\dfrac{\max_{1\le j\le n}\left|\sum_{i=1}^{j}X_i\right|}{b_n}\overset{\mathbb{P}}{\to}0\ \text{ as }n\to\infty,
		\end{split}
	\end{equation}
	irrespective of the joint distribution of the $X_n$'s.
	Boukhari \cite{boukhari2021weak} presented an example showing that his result
	does not hold when $p=1$. 
The proof of
the sufficient part of Theorem \ref{thm.gut} in \cite{gut2004extension} works well with pairwise independent random variables
since we do not involve the maximal partial sums. 
Krulov \cite{kruglov2011generalization} and Chandra \cite{chandra2014extension}
established WLLNs for maximal partial sums for the case where the summands are negatively associated and asymptotically almost negatively associated, respectively.
	The authors in \cite{chandra2014extension,kruglov2011generalization} considered general normalizing constants, which
	showed that the sufficient part of Theorem \ref{thm.gut} also holds for $1\le p<2$.
	However,  the method used in \cite{chandra2014extension,kruglov2011generalization} requires
a Kolmogorov-type maximal inequality (see Lemma 1.2 in \cite{chandra2014extension}) which does not hold for pairwise independent random variables.
	
	The aim of this paper is to establish WLLNs for maximal partial sums of pairwise independent random variables thereby
	extending the sufficient part of 
	Theorem \ref{thm.gut} for the case $p=1$ to WLLN for maximal partial sums from sequences of pairwise independent random variables. 
	We use a technique developed by Rio \cite{rio1995vitesses} to avoid using the Kolmogorov
	maximal inequality.
In addition, we also establish
a WLLN for maximal partial sums
of pairwise independent random variables
under a uniform integrability condition,
and present an example to show that this result does not hold in general if the 
uniform integrability assumption is weakened to the uniform boundedness of the moments.

	Let $\Lambda$ be a nonempty index set.
	A family of random variables $\{X_{\lambda}, \lambda \in \Lambda \}$ is said to be \textit{stochastically
		dominated} by a random variable $X$ if 
	\begin{equation}\label{RT1}
		\sup_{\lambda\in\Lambda} \mathbb{P}(|X_{\lambda}|>t)\le  \mathbb{P}(|X|>t)\mbox{ for all } t\ge 0. 
	\end{equation}
	Some authors use an apparently weaker definition of $\{X_{\lambda}, \lambda \in \Lambda \}$ being
	stochastically dominated by a random variable $Y$, namely that
	\begin{equation}\label{RT2}
		\sup_{\lambda\in\Lambda} \mathbb{P}(|X_{\lambda}|>t)\le  C_1\mathbb{P}(C_2|Y|>t) \mbox{ for all } t\ge 0
	\end{equation}
	for some constants $C_1, C_2\in (0,\infty)$.
	It is shown recently by Rosalsky and Th\`{a}nh \cite{rosalsky2021note} that
	\eqref{RT1} and \eqref{RT2} are indeed equivalent. We note that if \eqref{RT1} is satisfied, then for all $t>0$ and $r>0$
	\[	\sup_{\lambda\in\Lambda} \E(|X_{\lambda}|^r\mathbf{1}(|X_\lambda|\le t))\le \E(|X|^r\mathbf{1}(|X|\le t))+t^r \mathbb{P}\{|X|>t\},\]
	and
	\[	\sup_{\lambda\in\Lambda} \E(|X_{\lambda}|^r\mathbf{1}(|X_\lambda|> t))\le \E(|X|^r\mathbf{1}(|X|> t)).\]

	The following theorem is the main result of this paper.
	\begin{theorem}\label{thm.main}
		Let $1\le p<2$ and let
		$\{X_n, \, n \geq 1\}$ be a sequence of pairwise independent random variables which is stochastically dominated by a random variable $X$.
		Let $L(\cdot)$ be a slowly varying function and let $b_n=n^{1/p}L(n)$, $n\ge1$.
		If	\begin{equation}\label{eq.main01}
			\lim_{n\to \infty}n\mathbb{P}(|X|>b_n)=0,
		\end{equation}
		then
		\begin{equation}\label{eq.main03}
			\begin{split}
				\dfrac{\max_{1\le j\le n}\left|\sum_{i=1}^{j}\left(X_i-\E\left(X_i\mathbf{1}(|X_i|\le b_n)\right)\right)\right|}{b_n}\overset{\mathbb{P}}{\to}0\ \text{ as }n\to\infty.
			\end{split}
		\end{equation}
	\end{theorem}
	
	We postpone the proof of Theorem \ref{thm.main} to Section \ref{sec:proofs}.
	From Theorem 3.2 of Boukhari \cite{boukhari2021weak}, we have that if $\{X_n,n\ge1\}$ is a sequence of
	pairwise independent random variables, and $\{b_n,n\ge1\}$ is a sequence of
	positive constants, then 
	\begin{equation}\label{boukhari.eq05}
		\frac{\max_{1\le i\le n}|X_i|}{b_n}\overset{\mathbb{P}}{\to} 0 \text{ if and only if } \sum_{i=1}^n \mathbb{P}(|X_i|>b_n\varepsilon)\to 0 \text{ for all }\varepsilon>0.
	\end{equation}
	
	By using Theorem \ref{thm.main} and \eqref{boukhari.eq05}, we obtain the following corollary.
	
	\begin{corollary}\label{thm.main3}
		Let $\{X,X_n, \, n \geq 1\}$ be a sequence of pairwise i.i.d. random variables.
		Let $p$, $L(\cdot)$, $b_n$ be as in Theorem \ref{thm.main}. Then \eqref{eq.main01} and \eqref{eq.main03} are equivalent.
	\end{corollary}
	\begin{proof}
		If \eqref{eq.main01} holds, then \eqref{eq.main03} follows immediately from Theorem \ref{thm.main}.	Now, assume that \eqref{eq.main03} holds.
		By the symmetrization procedure, it suffices to check the case where the random variables $X_n$, $n\ge 1$ are symmetric. In this case,  
		\eqref{eq.main03} becomes
		\begin{equation}\label{eq.gut.main03.symmetric}
			\dfrac{\max_{1\le j\le n}\left|S_j\right|}{b_n}\overset{\mathbb{P}}{\to}0\ \text{ as }n\to\infty,
		\end{equation} 
		where $S_j=\sum_{i=1}^{j}X_i,j\ge1$. Putting $S_0=0$, and applying \eqref{eq.gut.main03.symmetric} and inequality 
		\[\max_{1\le j\le n}|X_j|\le \max_{1\le j\le n}|S_j|+\max_{1\le j\le n}|S_{j-1}|,\]
		we obtain
		\begin{equation}\label{eq.main06}
			\dfrac{\max_{1\le j\le n}\left|X_j\right|}{b_n}\overset{\mathbb{P}}{\to}0\ \text{ as }n\to\infty.
		\end{equation}
		By combining \eqref{boukhari.eq05} and \eqref{eq.main06}, and using the identical distribution assumption, we obtain \eqref{eq.main01}.
	\end{proof}
	
	Theorem \ref{thm.main} also enables us to establish
	a WLLN for maximal partial sums
	of pairwise independent random variables
	under a uniform integrability condition.
After this paper was submitted, Th\`{a}nh \cite[Corollary 4.10]{thanh2022new} established a similar
	WLLN for triangular arrays of random variables satisfying a Kolmogorov-type maximal inequality. Theorem \ref{thm.sufficiency.for.stochastic.domination.5}
and Corollary 4.10 of Th\`{a}nh \cite{thanh2022new} do not imply each other.

Hereafter, 
	we denote the de Bruijn
	conjugate of a slowly varying function $L(\cdot)$
	by $\tilde{L}(\cdot)$.

	\begin{theorem}\label{thm.sufficiency.for.stochastic.domination.5}
		Let $1\le p<2$ and let $\{X_{n},n\ge 1\}$ be a sequence of pairwise independent random variables.
		Let $L(\cdot)$ be a slowly varying function.
		If $\{|X_{n}|^pL(|X_n|^p),n\ge 1\}$ is uniformly integrable,
		then
		\begin{equation}\label{ui.eq03}
			\begin{split}
				\dfrac{\max_{1\le j\le n}\left|\sum_{i=1}^{j}\left(X_i-\E(X_i)\right)\right|}{b_n}\overset{\mathbb{P}}{\to}0\ \text{ as }n\to\infty,
			\end{split}
		\end{equation}
		where $b_n=n^{1/p}\tilde{L}^{1/p}(n)$, $n\ge1$.
	\end{theorem}
	
	\begin{proof}
		Let $f(x)=x^pL(x^p),\ g(x)=x^{1/p}\tilde{L}^{1/p}(x),\ x\ge0$.
Recalling that we have assumed, without loss of generality, that $f$ and $g$ are strictly increasing on $[0,\infty)$.
		From Lemma 2.5 in Anh et al. \cite{anh2021marcinkiewicz}, we have
		\[\lim_{x\to\infty}\frac{f(g(x))}{x}=1,\]
		and therefore
		\begin{equation}\label{ui.eq07}
			f(g(n))>n/2 \text{ for all large }n.
		\end{equation}
		By the de La Vall\'{e}e Poussin criterion for uniform integrability, 
		there exists a nondecreasing function $h$ defined on $[0,\infty)$ with $h(0)=0$ such that
		\begin{equation}\label{eq.ui.domi.11}
			\lim_{x\to\infty}\dfrac{h(x)}{x}=\infty,
		\end{equation}
		and
		\begin{equation}\label{eq.ui.domi.12}
\sup_{i\ge 1}\E(h(f(|X_i|)))=\sup_{i\ge 1}\E(h(|X_{i}|^pL(|X_i|^p)))<\infty.
		\end{equation}
		By using Theorem 2.5 (i) of Rosalsky and Th\`{a}nh \cite{rosalsky2021note}, \eqref{eq.ui.domi.12} implies that 
the sequence $\{X_{n},n\ge 1\}$ is stochastically dominated by a nonnegative random variable $X$ with
		distribution function
		\[F(x)=1-\sup_{i\ge 1}\mathbb{P}(|X_i|>x), \ x\in\R.\]
We thus have by \eqref{ui.eq07}, \eqref{eq.ui.domi.11}, \eqref{eq.ui.domi.12} and the Markov inequality that
		\begin{equation*}
			\begin{split}
				\lim_{n\to\infty}n\mathbb{P}\left(X>b_n\right)
				&=\lim_{n\to\infty}n\sup_{i\ge 1}\mathbb{P}(|X_i|>g(n))\\
				&= \lim_{n\to\infty}n\sup_{i\ge 1}\mathbb{P}(f(|X_i|)> f(g(n)))\\
				&\le \lim_{n\to\infty}n\sup_{i\ge 1}\mathbb{P}(f(|X_i|)\ge n/2)\\
				&\le \lim_{n\to\infty}n\sup_{i\ge 1}\mathbb{P}(h(f(|X_i|))\ge h(n/2))\\
				&\le \lim_{n\to\infty}n \sup_{i\ge 1} \dfrac{\E(h(f(|X_i|)))}{h(n/2))}\\
				&=2\sup_{i\ge 1}\E(h(f(|X_i|)))\lim_{n\to\infty}\dfrac{n/2}{h(n/2)}=0.
			\end{split}
		\end{equation*}
Applying Theorem \ref{thm.main}, we obtain
		\begin{equation}\label{ui.eq05}
		\begin{split}
			\dfrac{\max_{1\le j\le n}\left|\sum_{i=1}^{j}\left(X_i-\E\left(X_i\mathbf{1}(|X_i|\le b_n)\right)\right)\right|}{b_n}\overset{\mathbb{P}}{\to}0\ \text{ as }n\to\infty.
		\end{split}
	\end{equation}
By the same argument as in the proof of Corollary 4.10 of Th\`{a}nh \cite{thanh2022new} (see Equation (4.45) in \cite{thanh2022new}), we have
		\begin{equation}\label{ui.eq06}
	\begin{split}
		\dfrac{\max_{1\le j\le n}\left|\sum_{i=1}^{j} \E\left(X_i\mathbf{1}(|X_i|> b_n)\right)\right|}{b_n}\to 0\ \text{ as }n\to\infty.
	\end{split}
\end{equation}
Combining \eqref{ui.eq05} and \eqref{ui.eq06} yields \eqref{ui.eq03}.
\end{proof}
	
	The following example shows that in Theorem \ref{thm.sufficiency.for.stochastic.domination.5},
	the assumption that $\{|X_{n}|^pL(|X_n|^p),n\ge 1\}$ is uniformly integrable cannot be weakened to 
	\begin{equation}\label{bound.moment.eq03}
		\sup_{n\ge 1}\E(|X_{n}|^pL(|X_n|^p))<\infty.
	\end{equation}
	
	\begin{example}
		Let $1\le p<2$, and let
		$\{X_n,n\ge1\}$ be a sequence of
		independent symmetric random variables with
		\[\mathbb{P}(X_n=0)=1-\frac{1}{n},\ \mathbb{P}(X_n=n^{1/p})=\mathbb{P}(X_n=-n^{1/p})=\frac{1}{2n},\ n\ge 1. \]
		Consider the case where the slowly varying function $L(x)\equiv 1$. Then it is clear that
		\[\sup_{n\ge 1}\E(|X_{n}|^pL(|X_n|^p))=\sup_{n\ge 1}\E(|X_{n}|^p)=1<\infty\]
		and
		\[\sup_{n\ge 1}\E(|X_{n}|^pL(|X_n|^p)\mathbf{1}(|X_n|^pL(|X_n|^p)>a))=\sup_{n\ge 1}\E(|X_{n}|^p\mathbf{1}(|X_n|>a^{1/p}))=1\]
		for all $a>0$. Therefore \eqref{bound.moment.eq03} is satisfied but $\{|X_{n}|^pL(|X_n|^p),n\ge 1\}$ is not uniformly integrable.
		For a real number $x$, let $\lfloor x \rfloor$ denote the greatest integer that is smaller than or equal to $x$. Then for $0<\varepsilon<1/4$ and
		for $n\ge 2$, we have
		\begin{equation}\label{boukhari.eq07}
			\begin{split}
				\sum_{i=1}^n \mathbb{P}(|X_i|>\varepsilon n^{1/p})&\ge 	\sum_{i=\lfloor n/2 \rfloor}^n \mathbb{P}(|X_i|>\varepsilon n^{1/p})\\
				&\ge 	\sum_{i=\lfloor n/2 \rfloor}^n \frac{1}{n}\ge \frac{1}{2}.
			\end{split}
		\end{equation}
		Combining \eqref{boukhari.eq05} and \eqref{boukhari.eq07} yields
		\[	\frac{\max_{1\le i\le n}|X_i|}{n^{1/p}}\overset{\mathbb{P}}{\nrightarrow} 0 .\]
		This implies that \eqref{ui.eq03} (with $b_n\equiv n^{1/p}$) fails.
	\end{example}

	\section{Proof of Theorem \ref{thm.main}}\label{sec:proofs}
	
	The following lemma plays an important role in the proof of Theorem \ref{thm.main}. It gives a general approach to the WLLN.
	In this lemma, we do not require any dependence structure. Throughout this section, we use the symbol $C$ to denote a universal positive constant 
	which is not necessarily the same in
	each appearance.

	\begin{lemma}\label{lem.general}
		Let $\{b_n,n\ge1\}$ be a nondecreasing sequence of positive numbers satisfying
		\begin{equation}\label{lem.eq01}
			\sup_{m\ge1}\frac{b_{2^m}}{b_{2^{m-1}}}<\infty.
		\end{equation}
		Let $\{X_n,n\ge1\}$ be a sequence of
		random variables which is stochastically dominated by a random variable $X$ and let $X_{i,n}=X_i\mathbf{1}(|X_i|\le b_n),\ n\ge 1,i\ge 1$. Assume that
		\begin{equation}\label{lem.eq03}
			\lim_{n\to\infty} n\mathbb{P}(|X|>b_n)=0.
		\end{equation}
		Then
		\begin{equation}\label{lem.eq07}
			\dfrac{\max_{1\le j\le n}\left|\sum_{i=1}^{j}\left(X_{i}-\E\left(X_{i,n}\right)\right)\right|}{b_{n}}\overset{\mathbb{P}}{\to}0\ \text{ as }n\to\infty
		\end{equation}
		if and only if 
		\begin{equation}\label{lem.eq05}
			\dfrac{\max_{1\le j< 2^n}\left|\sum_{i=1}^{j}\left(X_{i,2^n}-\E\left(X_{i,2^n}\right)\right)\right|}{b_{2^n}}\overset{\mathbb{P}}{\to}0\ \text{ as }n\to\infty.
		\end{equation}
	\end{lemma}
	
	\begin{proof}
		We firstly prove under \eqref{lem.eq03} that
		\begin{equation}\label{lem.eq08}
			\begin{split}
				\dfrac{\max_{1\le j< 2^n}\left|\sum_{i=1}^{j}\left(X_i-X_{i,2^n}\right)\right|}{b_{2^n}}\overset{\mathbb{P}}{\to}0\ \text{ as }n\to\infty.
			\end{split}
		\end{equation}
		To see this, let $\varepsilon>0$ be arbitrary. Then
		\begin{equation*}
			\begin{split}
				\mathbb{P}\left(\dfrac{\max_{1\le j< 2^n}\left|\sum_{i=1}^{j}\left(X_i-X_{i,2^n}\right)\right|}{b_{2^n}}>\varepsilon\right)
				&\le \mathbb{P}\left(\bigcup_{i=1}^{2^n-1}\left(X_i\not=X_{i,2^n}\right) \right)\\
				&\le \sum_{i=1}^{2^n-1}\mathbb{P}\left(X_i\not=X_{i,2^n}\right)\\
				&\le 2^n \mathbb{P}\left(|X|>b_{2^n}\right)\to 0 \text{ as }n\to\infty \text{ (by \eqref{lem.eq03})}
			\end{split}
		\end{equation*}
		thereby proving \eqref{lem.eq08} since $\varepsilon>0$ is arbitrary.
		
		Next, assume that \eqref{lem.eq07} holds. Then
		\begin{equation}\label{lem.eq06}
			\dfrac{\max_{1\le j< 2^n}\left|\sum_{i=1}^{j}\left(X_{i}-\E\left(X_{i,2^n}\right)\right)\right|}{b_{2^n}}\overset{\mathbb{P}}{\to}0\ \text{ as }n\to\infty.
		\end{equation}
		Combining \eqref{lem.eq08} and \eqref{lem.eq06}, we obtain \eqref{lem.eq05}.
		
		Finally, assume that \eqref{lem.eq05} holds.
		It follows from \eqref{lem.eq05} and 
		\eqref{lem.eq08} that
		\begin{equation}\label{lem.eq09}
			\begin{split}
				\dfrac{\max_{1\le j< 2^n}\left|\sum_{i=1}^{j}\left(X_{i}-\E\left(X_{i,2^n}\right)\right)\right|}{b_{2^n}}\overset{\mathbb{P}}{\to}0\ \text{ as }n\to\infty.
			\end{split}
		\end{equation}
		Now, for $m\ge 1$, set
		\[K_m=\max_{2^{m-1}\le n<2^m}\frac{\max_{1\le j< 2^m}\left|\sum_{i=1}^j \E(X_{i,2^m}-X_{i,n})\right|}{b_{2^{m-1}}}.\]
		Then by using \eqref{lem.eq01}, \eqref{lem.eq03} and the stochastic domination assumption, we have
		\begin{equation}\label{lem.eq11}
			\begin{split}
				K_m&\le \dfrac{\sum_{i=1}^{2^m}  \mathbb{E}\left(|X_i|\mathbf{1}(b_{2^{m-1}}<|X_{i}|\le b_{2^{m}})\right)}{b_{2^{m-1}}}\\
				&\le \dfrac{\sum_{i=1}^{2^m} b_{2^m} \mathbb{P}(|X_{i}|>b_{2^{m-1}})}{b_{2^{m-1}}}\\
				&\le C 2^m \mathbb{P}(|X|>b_{2^{m-1}})\to 0 \text{ as } m\to\infty.
			\end{split}
		\end{equation}
		For $n\ge 1$, let $m\ge 1$ be such that $2^{m-1}\le n<2^{m}$. Then by \eqref{lem.eq01}, \eqref{lem.eq09} and \eqref{lem.eq11}), we have
		\begin{equation*}\label{lem.eq13}
			\begin{split}
				\frac{\max_{1\le j\le n}\left|\sum_{i=1}^j (X_i-\E(X_{i,n}))\right|}{b_n}&\le \dfrac{\max_{1\le j< 2^m}\left|\sum_{i=1}^j (X_i-\E(X_{i,2^m}))\right|}{b_{2^{m-1}}}\\
	&\qquad+\dfrac{\max_{1\le j< 2^m}\left|\sum_{i=1}^j \E(X_{i,2^m}-X_{i,n})\right|}{b_{2^{m-1}}}\\
				&\le \dfrac{C\max_{1\le j<2^m}\left|\sum_{i=1}^j (X_i-\E(X_{i,2^m}))\right|}{b_{2^{m}}}+K_m\\
				&\to 0 \text{ as } m\to\infty
			\end{split}
		\end{equation*}
		thereby establishing \eqref{lem.eq07}.
		
	\end{proof}
	
	\begin{proof}[Proof of Theorem \ref{thm.main}]
		Let
		\[X_{i,n}=X_i\mathbf{1}(|X_i|\le b_n),\ n\ge 1,i\ge 1.\] 
It is clear that the sequence $\{b_n,n\ge1\}$ satisfies \eqref{lem.eq01}. By Lemma \ref{lem.general}, it suffices to show that
		\begin{equation}\label{eq.main17}
			\begin{split}
				\dfrac{\max_{1\le j< 2^n}\left|\sum_{i=1}^{j}\left(X_{i,2^n}-\E\left(X_{i,2^n}\right)\right)\right|}{b_{2^n}}\overset{\mathbb{P}}{\to}0\ \text{ as }n\to\infty.
			\end{split}
		\end{equation}
		For $m\ge 0,$ set $S_{0,m}=0$ and
		\[S_{j,m}=\sum_{i=1}^j (X_{i,2^m}-\E X_{i,2^m}),\ j\ge 1.\]
		We will use techniques developed by Rio \cite{rio1995vitesses} (see also \cite{thanh2020theBaum} for the case of regular varying normalizing sequences) as follows.
		For $n\ge1$, $1\le j<2^n$ and for $0\le m\le n$, let $k=\lfloor j/2^m\rfloor $ be the greatest integer which is
		less than or equal to $j/2^m$. Then $0\le k<2^{n-m}$ and $k 2^m\le j< (k+1)2^m$. Let $j_m = k 2 ^m$, and
		\[Y_ {i, m} =\left|X_ {i, 2^m}-X_{i, 2^{m-1}}\right|-\E\left(\left|X_{i, 2^m}-X_{i, 2^{m-1}}\right|\right).\]
		Then we can show that (see \cite[Page 1236]{thanh2020theBaum})
		\begin{equation}\label{28}
			\begin{split}
				\max_{1\le j< 2^n}\left|S_{j,n}\right|
				&\le \sum_{m=1}^n \max_{0\le k<2^{n-m}}\left|\sum_{i=k2^m+1}^{k2^m+2^{m-1}}\left(X_{i, 2^{m-1}}-\E(X_{i, 2^{m-1}})\right)\right| +\sum_{m=1}^n \max_{0\le k<2^{n-m}}\left|\sum_{i=k2^m+1}^{(k+1)2^m}Y_{i,m}\right|\\
				&\quad +\sum_{m=1}^n 2^{m+1}
				b_{2^m}\mathbb{P}\left(|X|>b_{2^{m-1}}\right).
			\end{split}
		\end{equation}
		By \eqref{eq.main01}, we have
		\begin{equation}\label{29}
			2^{m+1} \mathbb{P}\left(|X|>b_{2^{m-1}}\right)\to 0 \text{ as }m\to\infty.
		\end{equation}
			It follows from Karamata's theorem (see, e.g., \cite[p. 30]{bingham1989regular}) that
	\begin{equation}\label{rem08}
		\sum_{k=1}^n \alpha^kL\left(\beta^k\right)\le C\alpha^nL(\beta^n)\text{ for all }\alpha>1,\beta\ge 1.
	\end{equation}
By using \eqref{29}, \eqref{rem08} and Toeplitz's lemma, we have
		\begin{equation}\label{34}
			\begin{split}
&\lim_{n\to \infty}\dfrac{\sum_{m=1}^n b_{2^m} 2^{m+1} \mathbb{P}\left(|X|>b_{2^{m-1}}\right)}{b_{2^n}}\\
&=\lim_{n\to \infty}\dfrac{\sum_{m=1}^n 2^{m/p} L(2^{m}) 2^{m+1} \mathbb{P}\left(|X|>b_{2^{m-1}}\right)}{2^{n/p} L(2^{n})}=0.
			\end{split}
		\end{equation}
		Let $\varepsilon_1>0$ be arbitrary, and let
		$a$ and $b$ be positive constants satisfying
		\[1/2<a<1/p,\ a+b=1/p.\]
		For $n\ge 1$, $0\le m\le n$,  let 
		\begin{equation}\label{eq.main24}
			\lambda_{m,n}=\varepsilon_{1}2^{bm}  2^{an}L(2^{n}).
		\end{equation}
An elementary calculation shows (see \cite[p. 1236]{thanh2020theBaum})
		\begin{equation}\label{30}
			\begin{split}
				\sum_{m=1}^n\lambda_{m,n}\le \dfrac{ 2^b \varepsilon_1 b_{2^{n}}}{2^b-1}:=C_1(b)\varepsilon_1 b_{2^{n}}.
			\end{split}
		\end{equation}
		By  \eqref{28}, \eqref{34}, the proof of \eqref{eq.main17} is completed if we show that
		\begin{equation}\label{38}
			I_n:=\mathbb{P}\left(\sum_{m=1}^n \max_{0\le k<2^{n-m}}\left|\sum_{i=k2^m+1}^{(k+1)2^m}Y_{i,m}\right|\ge C_1(b)\varepsilon_1 b_{2^{n}}\right)\to 0 \text{ as }n\to\infty,
		\end{equation}
		and
		\begin{equation}\label{39}
			J_n:=\mathbb{P}\left(\sum_{m=1}^n \max_{0\le k<2^{n-m}}\left|\sum_{i=k2^m+1}^{k2^m+2^{m-1}}\left(X_{i, 2^{m-1}}-\E(X_ {i, 2^{m-1}})\right)\right| \ge C_1(b)\varepsilon_1 b_{2^{n}}\right)\to 0 \text{ as }n\to\infty.
		\end{equation}
We note that for each $m\ge 1$, $Y_{i,m},i\ge1$ are mean $0$ and pairwise independent random variables. 
Therefore
		\begin{equation}\label{eq.main26}
			\begin{split}
				I_n&\le \sum_{m=1}^n\mathbb{P}\left(\max_{0\le k<2^{n-m}}\left|\sum_{i=k2^m+1}^{(k+1)2^m}Y_{i,m}\right|\ge \lambda_{m,n}\right)\text{ (by \eqref{30})}\\
				&\le \sum_{m=1}^n  \lambda_{m,n}^{-2} \E\left(\max_{0\le k<2^{n-m}}\left|\sum_{i=k2^m+1}^{(k+1)2^m}Y_{i,m}\right|\right)^2\text{ (by Markov's inequality)}\\
				&\le \sum_{m=1}^n \lambda_{m,n}^{-2} \sum_{k=0}^{2^{n-m}-1}\E\left(\sum_{i=k2^m+1}^{(k+1)2^m}Y_{i,m}\right)^2\\
				&= \sum_{m=1}^n \lambda_{m,n}^{-2} \sum_{k=0}^{2^{n-m}-1}\sum_{i=k2^m+1}^{(k+1)2^m}\E\left(Y_{i,m}\right)^2\\
				&\le \sum_{m=1}^n \lambda_{m,n}^{-2} \sum_{k=0}^{2^{n-m}-1}\sum_{i=k2^m+1}^{(k+1)2^m}\E\left(X_ {i, 2^m}-X_{i, 2^{m-1}}\right)^2\\
				&\le \sum_{m=1}^n 2^n \lambda_{m,n}^{-2}b_{2^m}^2\mathbb{P}(|X|>b_{2^{m-1}})\text{ (by the stochastic domination assumption)}\\
				&=\varepsilon_{1}^{-2}\frac{1}{2^{n (2a-1)}L^{2}(2^{n})}\left(\sum_{m=1}^n 2^{m(2a-1)}L^2(2^m)2^m\mathbb{P}(|X|>b_{2^{m-1}})\right)\text{ (by \eqref{eq.main24})}\\
				&\to 0 \text{ as } n\to\infty \text{ (by noting $2a-1>0$ and using \eqref{29}, \eqref{rem08}, and Toeplitz's lemma)}.
			\end{split}
		\end{equation}
Similarly,
		\begin{equation}\label{eq.main27}
			\begin{split}
				J_n
				&\le \sum_{m=1}^n\mathbb{P}\left(\max_{0\le k<2^{n-m}}\left|\sum_{i=k2^m+1}^{k2^m+2^{m-1}}\left(X_ {i, 2^{m-1}}-\E(X_ {i, 2^{m-1}})\right)\right| \ge \lambda_{m,n}\right)\\
				&\le  \sum_{m=1}^n 2^n \lambda_{m,n}^{-2}\left(\E X^2\mathbf{1} (|X|\le b_{2^{m-1}})+b_{2^{m-1}}^2\mathbb{P}(|X|>b_{2^{m-1}})\right)\\
				&=  \sum_{m=1}^n 2^n \lambda_{m,n}^{-2}\E X^2\mathbf{1} (|X|\le b_{2^{m-1}})+o(1),
			\end{split}
		\end{equation}
	where we have applied \eqref{eq.main26} in the final step. By using integration by parts, and proceeding in a similar manner as the last two lines of \eqref{eq.main26}, we have
		\begin{equation}\label{eq.main28}
			\begin{split}
&\sum_{m=1}^n 2^n \lambda_{m,n}^{-2}\E X^2\mathbf{1} (|X|\le b_{2^{m-1}})
\le \sum_{m=1}^n 2^n \lambda_{m,n}^{-2}\int_{0}^{b_{2^{m-1}}}2x\mathbb{P}(|X|>x)\dx x\\
				&\le  \varepsilon_{1}^{-2} 2^{n (1-2a)}L^{-2}(2^{n})\sum_{m=1}^n \left(2^{- 2mb}\sum_{k=1}^m \int_{b_{2^{k-1}}}^{b_{2^k}} 2x \mathbb{P}(|X|>x)\dx x +2^{- 2mb}\int_{0}^{b_1}2x\dx x\right)\\
				&\le \varepsilon_{1}^{-2}  2^{n (1-2a)}L^{-2}(2^{n}) \left(\sum_{k=1}^{n}\left(\sum_{m=k}^n 2^{-2bm}\right) b_{2^{k}}^2 \mathbb{P}\left(|X|>b_{2^{k-1}}\right)+\sum_{m=1}^n2^{- 2mb}b_{1}^2\right)\\
				&\le \dfrac{C}{2^{n (2a-1)}L^{2}(2^{n})}\left(\sum_{k=1}^{n} 2^{k(2a-1)} L^2(2^k) 2^k\mathbb{P}\left(|X|>b_{2^{k-1}}\right)+1\right)\to 0 \text{ as } n\to\infty.
			\end{split}
		\end{equation}
		Combining \eqref{eq.main26}--\eqref{eq.main28}, we obtain \eqref{38} and \eqref{39} thereby completing the proof of \eqref{eq.main17}. 
		
	\end{proof}
	
	\section{Concluding remarks}
	
	This paper establishes WLLNs for maximal partial sums of pairwise independent random variables
	without using the Kolmogorov maximal inequality.
	The method can be easily adapted to 
	dependent random variables. We have the following result:
	
	\begin{theorem}\label{thm.extension}
		Let $\{X_n,n\ge1\}$
		be a sequence of random variables and let $p$, $L(\cdot)$ and $b_n$ be as in Theorem \ref{thm.main}. Assume that there exists a constant $C$ such that
		for all nondecreasing functions $f_i$, $i\ge1$ we have
		\begin{equation}\label{eq:bound_var_00}
			\var\left(\sum_{i=k+1}^{k+\ell}f_{i}(X_{i})\right)\le C \sum_{i=k+1}^{k+\ell}\var(f_{i}(X_{i})),\ k\ge0,\ell\ge 1,
		\end{equation}
		provided the variances exist.
		If $\{X_n, \, n \geq 1\}$ is stochastically dominated by a random variable $X$ such that
		\eqref{eq.main01} is satisfied, then we obtain WLLN \eqref{eq.main03}.
	\end{theorem}
	
	Theorem \ref{thm.extension} can be proved by assuming that $X_n\ge 0$, $n\ge1$ since we can use identity $X_n=X_{n}^{+}-X_{n}^{-}$ 
	in the general case. We then use truncation $X_{i,n}=X_i\mathbf{1}(X_i\le b_n)+b_n\mathbf{1}(X_i>b_n),\ n\ge 1,i\ge 1$
	(to ensure that the truncated sequence
	 $\{X_{n,i},i\ge 1\}$ satisfies
\eqref{eq:bound_var_00}), and modify
the arguments given in
the proofs of Lemma \ref{lem.general} and Theorem \ref{thm.main}
	accordingly. The details are straightforward and, hence, are omitted.
	
	Many dependence structures satisfy \eqref{eq:bound_var_00}, including $m$-pairwise negative dependence, extended negative dependence, $\varphi$-mixing, etc (see, e.g., \cite{dzung2021thedependent,rio1995vitesses}).
	
It is obvious that \eqref{eq.main03} implies \eqref{eq.gut.main03}.	For the i.i.d case and $1\le p<2$, Kruglov \cite[Theorem 2]{kruglov2011generalization} proved that
\eqref{eq.main03} and \eqref{eq.gut.main03} are equivalent. It is an open problem as to whether or not this also holds for the pairwise i.i.d. case.
	
	\section*{Acknowledgements}
The author would like to thank the Editor for offering useful comments and suggestions which enabled him to improve the
paper. The author is also grateful to Professor Fakhreddine Boukhari for his interest in this work.
	
	\bibliographystyle{amsplain}
	\bibliography{mybib}
\end{document}